\documentclass[11pt,a4paper]{article}

\usepackage{amsmath,amssymb,amscd,amsthm}
\usepackage{graphics}
\usepackage{enumerate}
\input xy
\xyoption{all}

\usepackage{comment}
\usepackage{tikz}
\usepackage{setspace}

\usetikzlibrary{arrows}

\usepackage[paperwidth=210mm,
            paperheight=297mm,
            a4paper,
            twoside,
            top=1in, left=1.2in, right=1.2in, bottom=1in]{geometry}

\usepackage[active]{srcltx}

\newtheorem{lemma}{Lemma}

\newtheorem{thm}[lemma]{Theorem}

\newtheorem{rmk}[lemma]{Remark}

\newcommand{\ol}{\overline}
\newcommand{\hstar}{\mathfrak{h}^{*}}
\newcommand{\bbc}{\mathbb{C}}
\newcommand{\uh}{\mathcal{U}(\mathfrak{h})}
\newcommand{\uhn}{\mathcal{U}(\mathfrak{h}_n)}

\newcommand{\bb}[1]{\mathbb{#1}}
\newcommand{\fr}[1]{\mathfrak{#1}}
\newcommand{\ca}[1]{\mathcal{#1}}

\def\a{\mathbf{a}}
\def\e{\epsilon}
\usepackage{float}
\restylefloat{figure}

\allowdisplaybreaks

\begin{document}

\title{Generalized Verma modules over $\fr{sl}_{n+2}$ induced from $\ca{U}(\fr{h}_n)$-free $\fr{sl}_{n+1}$-modules}
\author{Yan-an Cai, Genqiang Liu, Jonathan Nilsson, Kaiming Zhao}
\date{}
\maketitle

\begin{abstract}
\noindent A class of generalized Verma modules over $\fr{sl}_{n+2}$ is constructed from  $\fr{sl}_{n+1}$-modules which are  $\uhn$-free modules of rank $1$.
 The necessary and sufficient conditions for these $\fr{sl}_{n+2}$-modules to be simple are determined. This leads to a class of  new simple $\fr{sl}_{n+2}$-modules.
\end{abstract}

\section{Introduction}
Classification of simple modules is an important step in the study of a module category over an algebra $\fr{g}$.
When $\fr{g}$ is a nontrivial complex  Lie algebra this turns out to be a very difficult problem in general, only for $\fr{g}=\fr{sl}_{2}$ (and several related small dimensional Lie algebras) there is such a classification to some extent, see~\cite{Bl,AP,Maz3}.
 Nevertheless, some subcategories of modules are relatively well understood, and their study has
occupied many mathematicians over the last century.

The most thoroughly studied class of modules are the \emph{weight modules} with finite dimensional weight spaces, which are modules on which a Cartan subalgebra $\fr{h} \subset \fr{g}$ acts diagonally.
Finite dimensional weight modules for complex finite dimensional semisimple Lie algebras were classified by Cartan already in 1913, see~\cite{Ca}.
The study of infinite dimensional weight modules was started by Verma and others in the 1960's (see~\cite{Ve}) and various results on weight modules were gathered over the years,
 see for example~\cite{Di,Hu,BGG,LLZ}. This eventually led to the classification of simple weight modules over semisimple Lie algebras with finite dimensional weight spaces in 2000, see~\cite{Fe,Fu,Mat}.
Some other classes of modules are also worth mentioning. These include Kostant's \emph{Whittaker modules}, on which the algebra of positive roots act locally finitely, see~\cite{Kos}.
 Also see~\cite{OW,BM,MZ,MW,Ni3} for a some related classes of modules. \emph{Gelfand-Zetlin} modules form a large family of modules parametrized by a type of triangular tableaux, see~\cite{DFO1,DFO2,Maz2,FGR}.
More recently the category of \emph{$\uh$-free modules} was defined, its objects are modules whose restriction to the Cartan subalgebra is free of rank $1$.
Simple $\uh$-free modules were classified for all finite-dimensional simple Lie algebras, see~\cite{Ni1,Ni2}, and also for some infinite dimensional Lie algebras, see~\cite{TZ1,TZ2}.

If $\fr{p}$ is a parabolic subalgebra of $\fr{g}$ with Levi-decomposition $\fr{p}=\fr{a}' \oplus \fr{n}$ and $V$ is an $\fr{a'}$-module with $\fr{n} V =0$, one usually writes
\[M_{\fr{p}}(V):={\text{Ind}}_{\fr{p}}^{\fr{g}}V=\mathcal{U}(\fr{g}) \otimes_{\mathcal{U}(\fr{p})} V\]
for the corresponding \emph{generalized Verma module} (GVM). GVM's have been extensively studied, especially under the condition that $V$ is a weight module, see for example~\cite{Maz1,BFL} and references therein.
In 1985 McDowell studied the GVM $M_{\fr{p}}(V)$ when $V$ is a simple non-degenerate Whittaker module, and he managed to determine necessary and sufficient conditions for $M_{\fr{p}}(V)$ to be simple, see~\cite{McD1,McD2}.
For non-weight modules $V$, there are also some more theoretical results on the structure of $M_{\fr{p}}(V)$, see~\cite{MaSt1,MiSo1,MiSo2,KM1,KM2}.

The present paper is inspired by McDowell's results. Consider $\fr{sl}_{n+1}$ as a subalgebra of $\fr{sl}_{n+2}$ in the top left corner.
  Let $\fr{p} \subset \fr{sl}_{n+2}$ be the parabolic subalgebra consisting of the standard Borel subalgebra plus $\fr{sl}_{n+1}$. Let $V$ be any  $\fr{sl}_{n+1}$-module which is
  a $\uhn$-free module of rank $1$. Using the classification
in~\cite{Ni1} (see Lemma \ref{classi}), $V$ can be extended to a $\fr{p}$-module labeled $V(\a,S,b,\lambda)$ for $ b\in\bbc,\lambda\in\bbc, \a=(a_1,\cdots,a_n,1)\in(\bbc^*)^{n+1}, S\subseteq \{1,2,...,n+1\}$.
In this setting our main result can be stated as follows (see Theorem \ref{mainthm} for all the details).

\begin{thm}
\label{introthm}
The induced module $M_{\fr{p}}(V(\a,S,b,\lambda))$ is simple if and only if  all the following conditions hold
  \begin{itemize}\item[(i).] $b-\lambda+2\not\in\bb{N}$, \item [(ii).] $nb+\lambda-1\not\in-\bb{N}$, \item[(iii).] $1 \leq |S| \leq n$ or $(n+1)b\notin \bb{Z}_+$.\end{itemize}
\end{thm}

Our paper is organized as follows:
In Section~\ref{section2} we define our setup, obtain some preliminary results on the structure of $M_{\fr{p}}(V(\a,S,b,\lambda))$, and give some technical lemmas.
Section~\ref{section3} is dedicated to the proof of our main theorem. We first prove that every nonzero element of
 $M_{\fr{p}}(V(\a,S,b,\lambda))$ can be reduced to a nonzero element of $1 \otimes V(\a,S,b,\lambda)$ under the conditions for simplicity.
Necessity of the conditions for simplicity is then proved by constructing explicit submodules
of $M_{\fr{p}}(V(\a,S,b,\lambda))$ for every $\lambda$ and $b$ such that either $b-\lambda+2$ or $-nb-\lambda+1$ is a natural number.

\begin{rmk}
Every annihilator of a simple module $L$ is also the annihilator of a simple highest weight module $L(\mu)$ for some $\mu \in \hstar$.
Moreover, a result of the paper \cite{MaSt1} states that if $Ann(L)=Ann(L(\mu))$ for two simple $\fr{p}$-modules $L$
 and $L(\mu)$, then $Ind_{\fr{p}}^{\fr{g}} L$ is simple if and only if $Ind_{\fr{p}}^{\fr{g}} L(\mu)$ is.
Thus an alternative way to prove the simplicity of our $M_{\fr{p}}(V)$ above would be to find out which $\mu$'s corresponds to our $\uh$-free modules, and then use Kazhdan--Lusztig combinatorics to show that
the highest weight module $Ind_{\fr{p}}^{\fr{g}} L(\mu)$ is simple for exactly these $\mu$. However, both of these steps seem difficult at the moment.\\
\end{rmk}

\noindent
{\bf Acknowledgements.}
The majority of this work was conducted during the first three authors' visit to the fourth author at Wilfrid Laurier University, Waterloo, in the summer of 2015.
The first three authors gratefully appreciates the hospitality of Professor Kaiming Zhao at Laurier.
G. Liu is partially supported by NSF of China (Grant
11301143) and  the school fund of Henan University (2012YBZR031, yqpy20140044);  J. Nilsson was partially supported by The Royal Swedish Academy of Sciences;
K. Zhao is partially supported by  NSF of China (Grants 11271109, 11471233) and NSERC.
We also thank Volodymyr Mazorchuk for his comments related to the connection to the paper~\cite{MaSt1}.

\section{Technical lemmas}
\label{section2}
Throughout the paper, we denote by $\bbc,\bb{N},\bb{Z}_+$ the sets of all complex numbers, positive integers and nonnegative integers, respectively. For  a set $S$ we define the indicator functions
$$\delta_{s\in S}=\begin{cases}&1 \text{ if } s\in S\\ &0 \text{ if } s\notin S\end{cases}, \  \  \delta_{s\notin S}=\begin{cases}&0 \text{ if } s\in S\\ &1 \text{ if } s\notin S\end{cases}.
$$
In this section we collect some of the basic definitions and establish technical results needed for studying our modules.

Let $\fr{g}$ be a semisimple complex finite dimensional Lie algebra with a fixed triangular decomposition $\fr{g}=\fr{n}_-\oplus\fr{h}\oplus\fr{n}_+$ and $\fr{p}\supset\fr{h}\oplus\fr{n}_+$ a parabolic subalgebra of $\fr{g}$ with the Levi decomposition $\fr{p}=(\fr{a}\oplus\fr{h}_\fr{a})\oplus\fr{n}$, where $\fr{n}$ is nilpotent, $\fr{a}'=\fr{a}\oplus\fr{h}_\fr{a}$ is reductive, $\fr{a}$ is semisimple and $\fr{h}_\fr{a}\subset\fr{h}$ is abelian and central in $\fr{a}'$. A {\it generalized Verma module} over $\fr{g}$ is an induced module
\[
M_\fr{p}(V)=\ca{U}(\fr{g})\otimes_{\ca{U}(\fr{p})}V,
\]
where $V$ is an $\fr{a}'$-module and $\fr{n}V=0$.

Consider the algebra $\fr{sl}_{n+2}(\bbc)$. Denote by $e_{i,j}(1\leq i,j\leq n+2)$ the $(n+2)\times(n+2)$ matrix with zeros everywhere except a $1$ on position $(i,j)$. For $1\leq k\leq n+2$, let $$h_{k}:=e_{k,k}-\frac{1}{n+1}\sum_{i=1}^{n+1}e_{i,i}.$$ Note that $h_{n+1}=-h_1-h_2-\cdots -h_{n}$. Then
\[  \{e_{i,j}| 1 \leq i\neq j \leq n+2\} \cup \{h_{1}, \ldots h_{n}, h_{n+2}\} \]
is a basis for $\fr{sl}_{n+2}(\bbc)$, and
\[  \{e_{i,j}| 1 \leq i\neq j \leq n+1\} \cup \{h_{1}, \ldots h_{n}\} \]
is a basis for $\fr{sl}_{n+1}(\bbc)$. We see that   $[\mathfrak{sl}_{n+1},h_{n+2}]=0$. The subspace spanned by $\{h_{1}, \ldots h_{n}\}$  is the standard Cartan subalgebra $\fr{h}_n$ of $\fr{sl}_{n+1}(\bbc)$. With respect to the basis, the bracket operation is given by the following lemma.
\begin{lemma}
\label{lemma1}
For $1\leq i\neq j\leq n+2, 1\leq i'\neq j'\leq n+2$ and $1\leq k,k'\leq n+2$, we have
\begin{align*}
[e_{i,j},e_{i',j'}] & = \delta_{j,i'}e_{i,j'} - \delta_{i,j'}e_{i',j},\\
[h_{k},e_{i,j}]&= (\delta_{k,i}-\delta_{k,j} + \frac{1}{n+1}(\delta_{i,n+2}-\delta_{j,n+2}))e_{i,j},\\
[h_{k},h_{k'}]&=0.
\end{align*}
\end{lemma}
We denote by $\ca{U}(\mathfrak{g})$ the universal enveloping algebra of a Lie algebra $\mathfrak{g}$. From the above lemma, we can prove the following useful formulas.

\begin{lemma}
\label{lemma2}
In $\ca{U}(\mathfrak{sl}_{n+2})$ we have the following relations: for $1 \leq i\neq k \leq n+1$
\begin{align*}
 e_{i,n+2}e_{n+2,k}^{m}&= e_{n+2,k}^{m}e_{i,n+2}+me_{n+2,k}^{m-1}e_{i,k}, \\
 e_{i,n+2}e_{n+2,i}^{m}&= e_{n+2,i}^{m}e_{i,n+2}+me_{n+2,i}^{m-1}(h_{i}-h_{n+2}-m+1), \\
 e_{i,k}e_{n+2,i}^m&= e_{n+2,i}^me_{i,k}-me_{n+2,k}e_{n+2,i}^{m-1}, \\
 h_ke_{n+2,i}^m&= e_{n+2,i}^m(h_k+\frac{m}{n+1}), \\
 h_ie_{n+2,i}^m&= e_{n+2,i}^m(h_i-\frac{mn}{n+1}),\\
 h_{n+2}e_{n+2,i}^m&= e_{n+2,i}^m(h_{n+2}+\frac{n+2}{n+1}m).
\end{align*}
\end{lemma}

\begin{proof}
We prove this lemma by induction on $m$. For $m=1$, the equations are just the first two equations in Lemma \ref{lemma1}. Now suppose that the equations hold for $m$, then for $1\leq i\neq k\leq n+1$,
\begin{align*}
e_{i,n+2}e_{n+2,k}^{m+1}&=e_{i,n+2}e_{n+2,k}^{m}e_{n+2,k}\\
&=(e_{n+2,k}^{m}e_{i,n+2}+me_{n+2,k}^{m-1}e_{i,k})e_{n+2,k}\\
&=e_{n+2,k}^{m}e_{i,n+2}e_{n+2,k}+me_{n+2,k}^{m}e_{i,k}\\
&=e_{n+2,k}^{m}(e_{n+2,k}e_{i,n+2}+e_{i,k})+me_{n+2,k}^{m}e_{i,k}\\
&=e_{n+2,k}^{m+1}e_{i,n+2}+(m+1)e_{n+2,k}^{m}e_{i,k};\\
\\
e_{i,n+2}e_{n+2,i}^{m+1}&=e_{i,n+2}e_{n+2,i}^{m}e_{n+2,i}\\
&=[e_{n+2,i}^{m}e_{i,n+2}+me_{n+2,i}^{m-1}(h_{i}-h_{n+2}-m+1)]e_{n+2,i}\\
&=e_{n+2,i}^{m}e_{i,n+2}e_{n+2,i}+me_{n+2,i}^{m-1}(h_{i}-h_{n+2}-m+1)e_{n+2,i}\\
&=e_{n+2,i}^{m}(e_{n+2,i}e_{i,n+2}+h_i-h_{n+2})+me_{n+2,i}^{m}(h_i-h_{n+2}-m-1)\\
&=e_{n+2,i}^{m+1}e_{i,n+2}+(m+1)e_{n+2,i}^{m}(h_{i}-h_{n+2}-m);\\
\\
e_{i,k}e_{n+2,i}^{m+1}&=e_{i,k}e_{n+2,i}^me_{n+2,i}\\
&=e_{n+2,i}^me_{i,k}e_{n+2,i}-me_{n+2,k}e_{n+2,i}^m\\
&=e_{n+2,i}^{m+1}e_{i,k}-e_{n+2,i}^{m}e_{n+2,k}-me_{n+2,k}e_{n+2,i}^m\\
&=e_{n+2,i}^{m+1}e_{i,k}-(m+1)e_{n+2,k}e_{n+2,i}^{m};\\
\\
h_ke_{n+2,i}^{m+1}&= h_ke_{n+2,i}^me_{n+2,i}\\
&=e_{n+2,i}^m(h_k+\frac{m}{n+1})e_{n+2,i}\\
&=e_{n+2,i}^{m+1}(h_k+\frac{1}{n+1}+\frac{m}{n+1})\\
&=e_{n+2,i}^{m+1}(h_k+\frac{m+1}{n+1});\\
\\
h_ie_{n+2,i}^{m+1}&=h_ie_{n+2,i}^me_{n+2,i}\\
&=e_{n+2,i}^m(h_i-\frac{mn}{n+1})e_{n+2,i}\\
&=e_{n+2,i}^{m+1}(h_i-\frac{n}{n+1}-\frac{mn}{n+1})\\
&=e_{n+2,i}^{m+1}(h_i-\frac{n(m+1)}{n+1});\\
\\
h_{n+2}e_{n+2,i}^{m+1}&=h_{n+2}e_{n+2,i}^me_{n+2,i}\\
&=e_{n+2,i}^m(h_{n+2}+\frac{n+2}{n+1}m)e_{n+2,i}\\
&=e_{n+2,i}^{m+1}(h_{n+2}+\frac{n+2}{n+1}+\frac{n+2}{n+1}m)\\
&=e_{n+2,i}^{m+1}(h_{n+2}+\frac{n+2}{n+1}(m+1)).
\end{align*}
\end{proof}

We recall the following fact about the classification of $\uhn$-free $\fr{sl}_{n+1}$-modules of rank~$1$.

\begin{lemma}
\label{classi}
 Let $S\subseteq\{1,2,\cdots,n+1\}, b\in\bbc$ and $\mathbf{a}=(a_1,\cdots,a_n,1)\in(\bbc^*)^{n+1}$.  Denote by $V(\a,S,b)$ the vector space $\bb{C}[h_{1}, \ldots, h_{n}]$
equipped with the $\fr{sl}_{n+1}$-module structure
\begin{align*}
h_i\cdot f&=h_if,  \\
e_{i,j}\cdot f&=a_ia_j^{-1}(\delta_{i\in S}+\delta_{i\not\in S}(h_i-b-1))(\delta_{j\in S}(h_j-b)+\delta_{j\not\in S})\sigma_{i}\sigma_j^{-1}(f),
\end{align*}
for all $f\in \bb{C}[h_{1}, \ldots, h_{n}]$ and  all $1\leq i\neq j\leq n+1$ where $\sigma_i$ $(1\leq i\leq n)$ is the algebra automorphism of $\bbc[h_1,\cdots,h_n]$ defined by mapping $h_k$ to $h_k-\delta_{k,i}$ while $\sigma_{n+1}$ is the identity map.
Then the set of  $\fr{sl}_{n+1}$-modules that are $\uhn$-free of rank $1$ is
\[\{V(\a,S,b) | S\subseteq\{1,2,\cdots,n+1\}, b\in\bbc,\mathbf{a}=(a_1,\cdots,a_n,1)\in(\bbc^*)^{n+1}\}.\]
   Moreover, $V(\a,S,b)$ is simple if and only if $1 \leq |S| \leq n$ or $(n+1)b \notin \bb{Z}_{+}$.
\end{lemma}

\begin{proof}
 This is a reformulation of Proposition~28, Theorem~29, and Proposition~31 of \cite{Ni1}.
\end{proof}
Now let $\fr{a}'=\fr{sl}_{n+1}(\bbc)+\bbc h_{n+2}$ (where $\fr{sl}_{n+1}(\bbc)$ is realized
as the upper left subalgebra of $\fr{sl}_{n+2}(\bbc)$) and let $M$ be an $\fr{a}'$-module whose restriction to $\fr{sl}_{n+1}(\bbc)$ is $\ca{U}(\fr{h}_n)$-free of rank 1. Since $h_{n+2}$ is central in $\fr{a}'$,
we know that $h_{n+2}$ acts on $M$ as a scalar $\lambda$ and $M$ is isomorphic to some $V(\a,S,b)$ as $\fr{sl}_{n+1}$-modules by the above lemma. Denote this $\fr{a}'$-module $M$ by $V(\a,S,b,\lambda)$.

Let $\fr{p}=\fr{a}'+\sum\limits_{i=1}^{n+1}\bbc e_{i,n+2}$. In this paper, we will study the generalized Verma module $M_\fr{p}(V(\a,S,b,\lambda))$ over $\fr{sl}_{n+2}$ induced from the simple $\fr{a}'$-module $V(\a,S,b,\lambda)$. Clearly, as vector spaces we have $$M_\fr{p}(V(\a,S,b,\lambda))\cong\bbc[e_{n+2,1},\cdots,e_{n+2,n+1}]\otimes\bbc[h_1,\cdots,h_n].$$

Note that $e_{n+2,1},\cdots,e_{n+2,n+1}$ commute. We consider elements of $M_\fr{p}(V(\a,S,b,\lambda))$ as polynomials in the $n+1$ variables $e_{n+2,1}, \ldots, e_{n+2,n+1}$ with coefficients (on the right side) in $\bb{C}[h_{1}, \ldots, h_{n}]$. To simplify notation we write $\ol{m}:=(m_{1}, \ldots ,m_{n+1}) \in \bb{Z}_+^{n+1}$,
and for $1 \leq k \leq n+1$ we let $\epsilon_{k}=(0, \ldots, 0, 1, 0, \ldots, 0)$ ($1$ in the $k$-th position).
Let $|\ol{m}|=m_{1}+ \cdots + m_{n+1}, \ol{m}!:=m_1!\cdots m_{n+1}!$. We also define
\begin{align*}
\a^{\ol{m}}&:=a_1^{m_1}\cdots a_{n+1}^{m_{n+1}}, \qquad \forall\, \a=(a_1,\cdots,a_{n+1})\in(\bbc^*)^{n+1},\\
E^{\ol{m}}&:=e_{n+2,1}^{m_{1}}e_{n+2,2}^{m_{2}} \cdots e_{n+2,n+1}^{m_{n+1}}.
\end{align*}
Any nonzero element $v$ in $M_{\fr{p}}(V(\a,S,b,\lambda))$ can be uniquely written as
\[
v=\sum\limits_{k=0}^{N}\sum\limits_{|\ol{m}|=k}E^{\ol{m}}P_{\ol{m}}=\sum\limits_{k=0}^Nv_k,
\]
where $P_{\ol{m}}\in\bbc[h_1,\cdots,h_n]$ and $v_k=\sum\limits_{|\ol{m}|=k}E^{\ol{m}}P_{\ol{m}}$ with $v_N\neq0$. We call $N$ the degree of $v$ and $v_k$ is homogeneous of degree $k$.
The following lemma tells us how the elements $e_{i,n+2}(1\leq i\leq n+1)$ act on an arbitrary homogeneous element in $M_{\fr{p}}(V(\a,S,b,\lambda))$ of degree $N$.
\begin{lemma}
\label{lemma3}
For fixed $N$, let
\[v=\sum_{|\ol{m}|=N} E^{\ol{m}}P_{\ol{m}},\] where $P_{\ol{m}} \in \bb{C}[h_{1}, \ldots, h_{n}]$.
For $1\leq i\leq n+1$, we then have
\[e_{i,n+2} \cdot v=\sum_{|\ol{m}|=N-1}E^{\ol{m}} \Big( \sum_{k\neq i}(m_{k}+1) (e_{i,k} \cdot P_{\ol{m}+\e_{k}})
+(m_{i}+1)(h_{i}-\lambda-N+1)P_{\ol{m}+\e_{i}} \Big).\]
\end{lemma}
\begin{proof}
We compute
\begin{align*}
 &e_{i,n+2} \cdot v= \sum_{|\ol{m}|=N} e_{i,n+2} \cdot E^{\ol{m}}P_{\ol{m}}\\
=&\sum_{|\ol{m}|=N} \Big(    e_{i,n+2} \cdot e_{n+2,i}^{m_{i}}(\prod_{j \neq i}e_{n+2,j}^{m_{j}})P_{\ol{m}}   \Big)\\
=&\sum_{|\ol{m}|=N} \Big(    (e_{n+2,i}^{m_{i}}e_{i,n+2}+m_{i}e_{n+2,i}^{m_{i}-1}(h_{i}-h_{n+2}-m_{i}+1))\cdot(\prod_{j \neq i}e_{n+2,j}^{m_{j}})P_{\ol{m}}   \Big)\\
=&\sum_{|\ol{m}|=N}     (e_{n+2,i}^{m_{i}}e_{i,n+2} \cdot(\prod_{j \neq i}e_{n+2,j}^{m_{j}})P_{\ol{m}})
+\sum_{|\ol{m}|=N}m_{i}e_{n+2,i}^{m_{i}-1}(h_{i}-h_{n+2}-m_{i}+1)\cdot(\prod_{j \neq i}e_{n+2,j}^{m_{j}})P_{\ol{m}}. \\
\end{align*}
 We compute the two sums separately:
\begin{align*}
\sum_{|\ol{m}|=N}&(e_{n+2,i}^{m_{i}}e_{i,n+2} \cdot(\prod_{j \neq i}e_{n+2,j}^{m_{j}})P_{\ol{m}})  \\
&=\sum_{|\ol{m}|=N}(e_{n+2,i}^{m_{i}} (\prod_{j \neq i}e_{n+2,j}^{m_{j}})e_{i,n+2}P_{\ol{m}})
+\sum_{|\ol{m}|=N}(e_{n+2,i}^{m_{i}} [e_{i,n+2},\prod_{k \neq i}e_{n+2,k}^{m_{k}}]P_{\ol{m}})  \\
&= 0  +\sum_{|\ol{m}|=N}e_{n+2,i}^{m_{i}}\sum_{k} [e_{i,n+2},e_{n+2,k}^{m_{k}}] (\prod_{j \neq i,k}e_{n+2,j}^{m_{j}})P_{\ol{m}}) \\
&=\sum_{|\ol{m}|=N}e_{n+2,i}^{m_{i}}\sum_{k} (m_{k}e_{n+2,k}^{m_{k}-1}e_{i,k}) (\prod_{j \neq i,k}e_{n+2,j}^{m_{j}})P_{\ol{m}}  \\
&=\sum_{|\ol{m}|=N}e_{n+2,i}^{m_{i}}\sum_{k} (m_{k}e_{n+2,k}^{m_{k}-1}) (\prod_{j \neq i,k}e_{n+2,j}^{m_{j}})(e_{i,k} \cdot P_{\ol{m}})  \\
&=\sum_{|\ol{m}|=N}\sum_{k\neq i} m_{k} E^{\ol{m}-\e_{k}}(e_{i,k} \cdot P_{\ol{m}})  \\
&=\sum_{k\neq i} \big (\sum_{|\ol{m}'|=N} m_{k}' E^{\ol{m'}-\e_{k}}(e_{i,k} \cdot P_{\ol{m'}}) \big). \\
\end{align*}
In the last step we just changed name for the summation variable $\ol{m}$.
We now introduce a new variable $\ol{m}:=\ol{m'}-\e_{k}$ in the inner sum. The above expression becomes
\begin{align*}
\sum_{k\neq i} \sum_{|\ol{m}|=N-1} (m_{k}+1) E^{\ol{m}}(e_{i,k} \cdot P_{\ol{m}+\e_{k}})=\sum_{|\ol{m}|=N-1}E^{\ol{m}} \Big( \sum_{k\neq i}(m_{k}+1) (e_{i,k} \cdot P_{\ol{m}+\e_{k}}) \Big). \\
\end{align*}

We now turn to the second term. Note first that for $1 \leq i\neq j \leq n+1$, we have $(h_{i}-h_{n+2})e_{n+2,j}=e_{n+2,j}(h_{i}-h_{n+2}-1)$. Using this relation
many times we can calculate

\begin{align*}
 \sum_{|\ol{m}|=N}&m_{i}e_{n+2,i}^{m_{i}-1}(h_{i}-h_{n+2}-m_{i}+1)\cdot(\prod_{j \neq i}e_{n+2,j}^{m_{j}})P_{\ol{m}}\\
&=\sum_{|\ol{m}|=N}m_{i}e_{n+2,i}^{m_{i}-1}(\prod_{j \neq i}e_{n+2,j}^{m_{j}})(h_{i}-h_{n+2}-m_{i}+1-\sum_{j\neq i}m_{j})P_{\ol{m}}\\
&=\sum_{|\ol{m}|=N}m_{i}E^{\ol{m}-\e_{i}}(h_{i}-\lambda-N+1)P_{\ol{m}}\\
&=\sum_{|\ol{m'}|=N}m_{i}'E^{\ol{m'}-\e_{i}}(h_{i}-\lambda-N+1)P_{\ol{m'}}.\\
\end{align*}
By another change of variables, $\ol{m}:=\ol{m'}-\e_{i}$, this simplifies to

\begin{align*}
\sum_{|\ol{m}|=N-1}&E^{\ol{m}} \big((m_{i}+1)(h_{i}-\lambda-N+1)P_{\ol{m}+\e_{i}}\big).\\
\end{align*}

Substituting any both of these into our original expression we can continue to simplify:

\begin{align*}
 &e_{i,n+2} \cdot v\\
=& (\sum_{|\ol{m}|=N-1}E^{\ol{m}} \Big( \sum_{k\neq i}(m_{k}+1) (e_{i,k} \cdot P_{\ol{m}+\e_{k}}) \Big)
+\sum_{|\ol{m}|=N-1}E^{\ol{m}} \big((m_{i}+1)(h_{i}-\lambda-N+1)P_{\ol{m}+\e_{i}}\big)\\
=&\sum_{|\ol{m}|=N-1}E^{\ol{m}} \Big( \sum_{k\neq i}(m_{k}+1) (e_{i,k} \cdot P_{\ol{m}+\e_{k}})
+(m_{i}+1)(h_{i}-\lambda-N+1)P_{\ol{m}+\e_{i}} \Big),
\end{align*}
which alternatively can be rewritten as
\begin{align*}
 \sum_{|\ol{m}|=N-1}E^{\ol{m}}  \sum_{k=1}^{n+1}(m_{k}+1)\big( (1-\delta_{i,k})(e_{i,k} \cdot P_{\ol{m}+\e_{k}})
+\delta_{i,k}(h_{i}-\lambda-N+1)P_{\ol{m}+\e_{k}}\big).
\end{align*}
This completes the proof.
\end{proof}

Let $\ol{m}\in\bb{Z}_+^{n+1}$ and $N\in\bb{N}$ with $N\geq|\ol{m}|$. For each $S\subseteq\{1,2,\cdots,n+1\}$ and $\a=(a_1,\cdots,a_n,1)\in(\bbc^*)^{n+1}$, define
\begin{align*}
P_{\ol{m}}(S,\a)=&\frac{\mathbf{a}^{\ol{m}}}{\ol{m}!}\prod\limits_{s\not\in S}\prod\limits_{k=1}^{m_s}(h_s-b-k),\\
P'_{\ol{m}}(S,\a)=&\frac{\mathbf{a}^{\ol{m}}}{\ol{m}!}\prod\limits_{\substack{s\not\in S\\ s\neq n+1}}\prod\limits_{k=1}^{m_s}(h_s-b-k)(\delta_{n+1\in S}\prod\limits_{r=1}^{N-m_{n+1}}(h_{n+1}-b-1+r)+\delta_{n+1\not\in S})\\
&\cdot\prod\limits_{t=1}^{m_{n+1}}(h_{n+1}+nb-t+1),\\
\Delta_{\ol{m}}(S,\a)=&\frac{\mathbf{a}^{\ol{m}}}{\ol{m}!}\prod\limits_{s\not\in S}\prod\limits_{k=1}^{m_s}(h_s-b-k)\prod\limits_{r=2}^{N-m_{n+1}}(h_{n+1}-b-2+r)\prod\limits_{t=1}^{m_{n+1}}(h_{n+1}+nb-t),\\
\Theta_{\ol{m}}(S,\a)=&\frac{\mathbf{a}^{\ol{m}}}{\ol{m}!}\prod\limits_{\substack{s\not\in S\\ s\neq n+1}}\prod\limits_{k=1}^{m_s}(h_s-b-k)\prod\limits_{k=1}^{m_{n+1}}(h_{n+1}-b-k-1),\\
\Upsilon_{\ol{m}}(S,\a)=&\frac{\mathbf{a}^{\ol{m}}}{\ol{m}!}\prod\limits_{\substack{s\not\in S\\ s\neq n+1}}\prod\limits_{k=1}^{m_s}(h_s-b-k)\prod\limits_{t=1}^{m_{n+1}}(h_{n+1}+nb-t).
\end{align*}

Then we have the following lemma.

\begin{lemma}
\label{lemma7}
Let $\ol{m}\in\bb{Z}_+^{n+1}$, $N\in\bb{N}$ with $N\geq|\ol{m}|$, $S\subseteq\{1,2,\cdots,n+1\}$ and $\a=(a_1,\cdots,a_n,1)\in(\bbc^*)^{n+1}$.
\begin{enumerate}[(1)]
\item If $n+1\in S$, then
\begin{align*}
a_j^{-1}(m_j+1)\sigma_j^{-1}(P_{\ol{m}+\e_j}(S,\a))=&\left\{\begin{array}{ll}
P_{\ol{m}}(S,\a), & 1\leq j\leq n+1, j\in S,\\
(h_j-b)P_{\ol{m}}(S,\a), & 1\leq j\leq n+1, j\not\in S;
\end{array}\right.\\
a_j^{-1}(m_j+1)\sigma_j^{-1}(P'_{\ol{m}+\e_j}(S,\a))=&\left\{\begin{array}{l}
(h_{n+1}-b-1)\Delta_{\ol{m}}(S,\a),  1\leq j\leq n, j\in S,\\
(h_j-b)(h_{n+1}-b-1)\Delta_{\ol{m}}(S,\a), 1\leq j\leq n, j\not\in S,\\
(h_{n+1}+nb)\Delta_{\ol{m}}(S,\a), j=n+1.
\end{array}\right.
\end{align*}
\item If $n+1\not\in S$, then
\begin{align*}
a_j^{-1}(m_j+1)\sigma_j^{-1}(P_{\ol{m}+\e_j}(S,\a))=&\left\{\begin{array}{l}
\Theta_{\ol{m}}(S,\a), \qquad 1\leq j\leq n, j\in S,\\
(h_j-b)\Theta_{\ol{m}}(S,\a), \qquad 1\leq j\leq n, j\not\in S,\\
(h_{n+1}-b-1)\Theta_{\ol{m}}(S,\a), \qquad j=n+1;
\end{array}\right.\\
a_j^{-1}(m_j+1)\sigma_j^{-1}(P'_{\ol{m}+\e_j}(S,\a))=&\left\{\begin{array}{ll}
\Upsilon_{\ol{m}}(S,\a), & 1\leq j\leq n, j\in S,\\
(h_j-b)\Upsilon_{\ol{m}}(S,\a), & 1\leq j\leq n, j\not\in S,\\
(h_{n+1}+nb)\Upsilon_{\ol{m}}(S,\a), & j=n+1.
\end{array}\right.\\
\end{align*}
\end{enumerate}
\end{lemma}

\begin{proof}
The lemma follows from direct computations.
\begin{enumerate}[(1)]
\item Let $n+1\in S$. For $1\leq j\leq n$ with $j\in S$, we have
\begin{align*}
&a_j^{-1}(m_j+1)\sigma_j^{-1}(P_{\ol{m}+\e_j}(S,\a))\\
=&a_j^{-1}(m_j+1)\sigma_j^{-1}\Big(\frac{\mathbf{a}^{\ol{m}+\e_j}}{(\ol{m}+\e_j)!}\prod\limits_{s\not\in S}\prod\limits_{k=1}^{m_s}(h_s-b-k)\Big)\\
=&\frac{\mathbf{a}^{\ol{m}}}{\ol{m}!}\prod\limits_{s\not\in S}\prod\limits_{k=1}^{m_s}(h_s-b-k)\\
=&P_{\ol{m}}(S,\a),\\
&\\
&a_j^{-1}(m_j+1)\sigma_j^{-1}(P'_{\ol{m}+\e_j}(S,\a))\\
=&a_j^{-1}(m_j+1)\sigma_j^{-1}\Big(\frac{\mathbf{a}^{\ol{m}+\e_j}}{(\ol{m}+\e_j)!}\prod\limits_{s\not\in S}\prod\limits_{k=1}^{m_s}(h_s-b-k)\prod\limits_{r=1}^{N-m_{n+1}}(h_{n+1}-b-1+r)\\
&\cdot\prod\limits_{t=1}^{m_{n+1}}(h_{n+1}+nb-t+1)\Big)\\
=&\frac{\mathbf{a}^{\ol{m}}}{\ol{m}!}\prod\limits_{s\not\in S}\prod\limits_{k=1}^{m_s}(h_s-b-k)\prod\limits_{r=1}^{N-m_{n+1}}(h_{n+1}-b-2+r)\prod\limits_{t=1}^{m_{n+1}}(h_{n+1}+nb-t)\\
=&\frac{\mathbf{a}^{\ol{m}}}{\ol{m}!}(h_{n+1}-b-1)\prod\limits_{s\not\in S}\prod\limits_{k=1}^{m_s}(h_s-b-k)\prod\limits_{r=2}^{N-m_{n+1}}(h_{n+1}-b-2+r)\\
&\cdot\prod\limits_{t=1}^{m_{n+1}}(h_{n+1}+nb-t)\\
=&(h_{n+1}-b-1)\Delta_{\ol{m}}(S,\a).
\end{align*}
For $1\leq j\leq n$ with $j\not\in S$, we have
\begin{align*}
&a_j^{-1}(m_j+1)\sigma_j^{-1}(P_{\ol{m}+\e_j}(S,\a))\\
=&a_j^{-1}(m_j+1)\sigma_j^{-1}\Big(\frac{\mathbf{a}^{\ol{m}+\e_j}}{(\ol{m}+\e_j)!}\prod\limits_{\substack{s\not\in S\\s\neq j}}\prod\limits_{k=1}^{m_s}(h_s-b-k)\prod\limits_{k=1}^{m_j+1}(h_j-b-k)\Big)\\
=&\frac{\mathbf{a}^{\ol{m}}}{\ol{m}!}\prod\limits_{\substack{s\not\in S\\ s\neq j}}\prod\limits_{k=1}^{m_s}(h_s-b-k)\prod\limits_{k=1}^{m_j+1}(h_j-b-k+1)\\
=&\frac{\mathbf{a}^{\ol{m}}}{\ol{m}!}(h_j-b)\prod\limits_{s\not\in S}\prod\limits_{k=1}^{m_s}(h_s-b-k)\\
=&(h_j-b)P_{\ol{m}}(S,\a),\\
&\\
&a_j^{-1}(m_j+1)\sigma_j^{-1}(P'_{\ol{m}+\e_j}(S,\a))\\
=&a_j^{-1}(m_j+1)\sigma_j^{-1}\Big(\frac{\mathbf{a}^{\ol{m}+\e_j}}{(\ol{m}+\e_j)!}\prod\limits_{\substack{s\not\in S\\ s\neq j}}\prod\limits_{k=1}^{m_s}(h_s-b-k)\prod\limits_{k=1}^{m_j+1}(h_j-b-k)\\
&\cdot\prod\limits_{r=1}^{N-m_{n+1}}(h_{n+1}-b-1+r)\prod\limits_{t=1}^{m_{n+1}}(h_{n+1}+nb-t+1)\Big)\\
=&\frac{\mathbf{a}^{\ol{m}}}{\ol{m}!}(h_j-b)\prod\limits_{s\not\in S}\prod\limits_{k=1}^{m_s}(h_s-b-k)\prod\limits_{r=1}^{N-m_{n+1}}(h_{n+1}-b-2+r)\prod\limits_{t=1}^{m_{n+1}}(h_{n+1}+nb-t)\\
=&\frac{\mathbf{a}^{\ol{m}}}{\ol{m}!}(h_j-b)(h_{n+1}-b-1)\prod\limits_{s\not\in S}\prod\limits_{k=1}^{m_s}(h_s-b-k)\prod\limits_{r=2}^{N-m_{n+1}}(h_{n+1}-b-2+r)\\
&\cdot\prod\limits_{t=1}^{m_{n+1}}(h_{n+1}+nb-t)\\
=&(h_j-b)(h_{n+1}-b-1)\Delta_{\ol{m}}(S,\a).
\end{align*}
Finally, for $j=n+1$, we have
\begin{align*}
&a_{n+1}^{-1}(m_{n+1}+1)\sigma_{n+1}^{-1}(P_{\ol{m}+\e_{n+1}}(S,\a))\\
=&a_{n+1}^{-1}(m_{n+1}+1)\frac{\mathbf{a}^{\ol{m}+\e_{n+1}}}{(\ol{m}+\e_{n+1})!}\prod\limits_{s\not\in S}\prod\limits_{k=1}^{m_s}(h_s-b-k)\\
=&\frac{\mathbf{a}^{\ol{m}}}{\ol{m}!}\prod\limits_{s\not\in S}\prod\limits_{k=1}^{m_s}(h_s-b-k)\\ =&P_{\ol{m} }(S,\a),\\
&
\\
&a_{n+1}^{-1}(m_{n+1}+1)\sigma_{n+1}^{-1}(P'_{\ol{m}+\e_{n+1}}(S,\a))\\
=&a_{n+1}^{-1}(m_{n+1}+1)\frac{\mathbf{a}^{\ol{m}+\e_{n+1}}}{(\ol{m}+\e_{n+1})!}\prod\limits_{s\not\in S}\prod\limits_{k=1}^{m_s}(h_s-b-k)\prod\limits_{r=1}^{N-m_{n+1}-1}(h_{n+1}-b-1+r)\\
&\cdot\prod\limits_{t=1}^{m_{n+1}+1}(h_{n+1}+nb-t+1)\\
=&\frac{\mathbf{a}^{\ol{m}}}{\ol{m}!}(h_{n+1}+nb)\prod\limits_{s\not\in S}\prod\limits_{k=1}^{m_s}(h_s-b-k)\prod\limits_{r=2}^{N-m_{n+1}}(h_{n+1}-b-2+r)\prod\limits_{t=1}^{m_{n+1}}(h_{n+1}+nb-t)\\ =&
(h_{n+1}+nb)\Delta_{\ol{m}}(S,\a).\\
\end{align*}
\item Let $n+1\not\in S$. For $1\leq j\leq n$ with $j\in S$, we have
\begin{align*}
&a_j^{-1}(m_j+1)\sigma_j^{-1}(P_{\ol{m}+\e_j}(S,\a))\\
=&a_j^{-1}(m_j+1)\sigma_j^{-1}\Big(\frac{\mathbf{a}^{\ol{m}+\e_j}}{(\ol{m}+\e_j)!}\prod\limits_{\substack{s\not\in S\\ s\neq n+1}}\prod\limits_{k=1}^{m_s}(h_s-b-k)\prod\limits_{k=1}^{m_{n+1}}(h_{n+1}-b-k)\Big)\\
=&\frac{\mathbf{a}^{\ol{m}}}{\ol{m}!}\prod\limits_{\substack{s\not\in S\\ s\neq n+1}}\prod\limits_{k=1}^{m_s}(h_s-b-k)\prod\limits_{k=1}^{m_{n+1}}(h_{n+1}-b-k-1)\\
=&\Theta_{\ol{m}}(S,\a),\\ & \\
&a_j^{-1}(m_j+1)\sigma_j^{-1}(P'_{\ol{m}+\e_j}(S,\a))\\
=&a_j^{-1}(m_j+1)\sigma_j^{-1}\Big(\frac{\mathbf{a}^{\ol{m}+\e_j}}{(\ol{m}+\e_j)!}\prod\limits_{\substack{s\not\in S\\ s\neq n+1}}\prod\limits_{k=1}^{m_s}(h_s-b-k)\prod\limits_{k=1}^{m_{n+1}}(h_{n+1}+nb-k+1)\Big)\\
=&\frac{\mathbf{a}^{\ol{m}}}{\ol{m}!}\prod\limits_{\substack{s\not\in S\\ s\neq n+1}}\prod\limits_{k=1}^{m_s}(h_s-b-k)\prod\limits_{k=1}^{m_{n+1}}(h_{n+1}+nb-k)\\
=&\Upsilon_{\ol{m}}(S,\a).
\end{align*}
For $1\leq j\leq n$ with $j\not\in S$, we have
\begin{align*}
&a_j^{-1}(m_j+1)\sigma_j^{-1}(P_{\ol{m}+\e_j}(S,\a))\\
=&a_j^{-1}(m_j+1)\sigma_j^{-1}\Big(\frac{\mathbf{a}^{\ol{m}+\e_j}}{(\ol{m}+\e_j)!}\prod\limits_{\substack{s\not\in S\\ s\neq j,n+1}}\prod\limits_{k=1}^{m_s}(h_s-b-k)\prod\limits_{k=1}^{m_j+1}(h_j-b-k)\\
&\cdot\prod\limits_{k=1}^{m_{n+1}}(h_{n+1}-b-k)\Big)\\
=&\frac{\mathbf{a}^{\ol{m}}}{\ol{m}!}\prod\limits_{\substack{s\not\in S\\ s\neq j,n+1}}\prod\limits_{k=1}^{m_s}(h_s-b-k)\prod\limits_{k=1}^{m_j+1}(h_j-b-k+1)\prod\limits_{k=1}^{m_{n+1}}(h_{n+1}-b-k-1),\\
=&\frac{\mathbf{a}^{\ol{m}}}{\ol{m}!}(h_j-b)\prod\limits_{\substack{s\not\in S\\ s\neq n+1}}\prod\limits_{k=1}^{m_s}(h_s-b-k)\prod\limits_{k=1}^{m_{n+1}}(h_{n+1}-b-k-1)\\
=&(h_j-b)\Theta_{\ol{m}}(S,\a),\\ &\\
&a_j^{-1}(m_j+1)\sigma_j^{-1}(P'_{\ol{m}+\e_j}(S,\a))\\
=&a_j^{-1}(m_j+1)\sigma_j^{-1}\Big(\frac{\mathbf{a}^{\ol{m}+\e_j}}{(\ol{m}+\e_j)!}\prod\limits_{\substack{s\not\in S\\ s\neq j,n+1}}\prod\limits_{k=1}^{m_s}(h_s-b-k)\prod\limits_{k=1}^{m_j+1}(h_j-b-k+1)\\
&\cdot\prod\limits_{k=1}^{m_{n+1}}(h_{n+1}+nb-k+1)\Big)\\
=&\frac{\mathbf{a}^{\ol{m}}}{\ol{m}!}\prod\limits_{\substack{s\not\in S\\ s\neq j,n+1}}\prod\limits_{k=1}^{m_s}(h_s-b-k)\prod\limits_{k=1}^{m_j+1}(h_j-b-k+1)\prod\limits_{k=1}^{m_{n+1}}(h_{n+1}+nb-k),\\
=&\frac{\mathbf{a}^{\ol{m}}}{\ol{m}!}(h_j-b)\prod\limits_{\substack{s\not\in S\\ s\neq n+1}}\prod\limits_{k=1}^{m_s}(h_s-b-k)\prod\limits_{k=1}^{m_{n+1}}(h_{n+1}+nb-k)\\
=&(h_j-b)\Upsilon_{\ol{m}}(S,\a).
\end{align*}
Now to complete the proof, it remains to verify the equations when $j=n+1\not\in S$.
\begin{align*}
&a_{n+1}^{-1}(m_{n+1}+1)\sigma_{n+1}^{-1}(P_{\ol{m}+\e_{n+1}}(S,\a))\\
=&a_{n+1}^{-1}(m_{n+1}+1)\frac{\mathbf{a}^{\ol{m}+\e_{n+1}}}{(\ol{m}+\e_{n+1})!}\prod\limits_{\substack{s\not\in S\\ s\neq n+1}}\prod\limits_{k=1}^{m_s}(h_s-b-k)\prod\limits_{k=1}^{m_{n+1}+1}(h_{n+1}-b-k)\\
=&\frac{\mathbf{a}^{\ol{m}}}{\ol{m}!}(h_{n+1}-b-1)\prod\limits_{\substack{s\not\in S\\ s\neq n+1}}\prod\limits_{k=1}^{m_s}(h_s-b-k)\prod\limits_{k=1}^{m_{n+1}}(h_{n+1}-b-k-1)\\
=&(h_{n+1}-b-1)\Theta_{\ol{m}}(S,\a),\\ &\\
&a_{n+1}^{-1}(m_{n+1}+1)\sigma_{n+1}^{-1}(P'_{\ol{m}+\e_{n+1}}(S,\a))\\
=&a_{n+1}^{-1}(m_{n+1}+1)\frac{\mathbf{a}^{\ol{m}+\e_{n+1}}}{(\ol{m}+\e_{n+1})!}\prod\limits_{\substack{s\not\in S\\ s\neq n+1}}\prod\limits_{k=1}^{m_s}(h_s-b-k)\prod\limits_{k=1}^{m_{n+1}+1}(h_{n+1}+nb-k+1)\\
=&\frac{\mathbf{a}^{\ol{m}}}{\ol{m}!}(h_{n+1}+nb)\prod\limits_{\substack{s\not\in S\\ s\neq n+1}}\prod\limits_{k=1}^{m_s}(h_s-b-k)\prod\limits_{t=1}^{m_{n+1}}(h_{n+1}+nb-t)\\
=&(h_{n+1}+nb)\Upsilon_{\ol{m}}(S,\a).
\end{align*}
\end{enumerate}
\end{proof}

\section{Proof of the main theorem}
\label{section3}
In this section, we will study the generalized Verma module $M_{\fr{p}}(V(\a,S,b,\lambda))$ over $\mathfrak{sl}_{n+2}$ for the given parameters $ b\in\bbc,\lambda\in\bbc, \a=(a_1,\cdots,a_n,1)\in(\bbc^*)^{n+1}, S\subseteq \{1,2,...,n+1\}$. Indeed, we will give sufficient and necessary conditions for $M_{\fr{p}}(V(\a,S,b,\lambda))$ to be simple. Before proving our main theorem, we need several auxiliary lemmas. First, we have

\begin{lemma}
\label{homlemma}
Let $W\subseteq M_{\fr{p}}(V(\a,S,b,\lambda))$ be a nonzero submodule. Suppose
\[v=\sum\limits_{k=0}^Nv_k\in W,\]
where $v_k$ is homogeneous of degree $k$, then $v_k\in W$  for all $k$.
\end{lemma}

\begin{proof}
Since
\[
h_{n+2}\cdot v_k=(\lambda+\frac{n+2}{n+1}k)v_k,
\]
we see that $v_k$ are weight vectors of different weights with respect to $h_{n+2}$.
Hence, $v_k\in W$ for all $k$.
\end{proof}

\begin{lemma}
\label{lemmasub}
Let $N\in \mathbb{N}$, and let $v=\sum\limits_{|\ol{m}|=N}E^{\ol{m}}P_{\ol{m}}$ be a nonzero homogeneous element in $M_{\fr{p}}(V(\a,S,b,\lambda))$ of degree $N$  where $P_{\ol{m}}\in\bbc[h_1,\cdots,h_n]$. If
\[
e_{i,n+2}\cdot v=0 \quad \text{ for all } \quad 1\leq i\leq n+1,
\]
then $\ca{U}(\fr{sl}_{n+2}(\bbc))v$ is a nonzero proper $\fr{sl}_{n+2}(\bbc)$-submodule of $M_{\fr{p}}(V(\a,S,b,\lambda))$.
\end{lemma}

\begin{proof}
Clearly, $\ca{U}(\fr{sl}_{n+2}(\bbc))v$ is a nonzero submodule of $M_{\fr{p}}(V(\a,S,b,\lambda))$. To show it is proper, we will show that any nonzero element in $\ca{U}(\fr{sl}_{n+2}(\bbc))v$ has degree greater than or equal to $N$. By the PBW Theorem we see that
$$\ca{U}(\fr{sl}_{n+2})=\bbc[e_{n+2,1},\cdots,e_{n+2,n+1}]\otimes\bbc[h_{n+2}] \otimes \ca{U}(\fr{sl}_{n+1}) \otimes \bbc[e_{1,n+2 },\cdots,e_{n+1,n+2}],
$$ yielding that
$$\ca{U}(\fr{sl}_{n+2})v=\bbc[e_{n+2,1},\cdots,e_{n+2,n+1}]\bbc[h_{n+2}]\ca{U}(\fr{sl}_{n+1})v.
$$
So we only need to show that if the following elements are not zero, then they are homogeneous of degree $N$: $h_{n+2}\cdot v, h_k\cdot v (1\leq k\leq n), e_{i,k}\cdot v (1\leq i\neq k\leq n+1)$.

Since for $1\leq i\neq k\leq n+1$, using the formulas in Lemma~\ref{lemma2} we have
\begin{align*}
h_{n+2}\cdot v&=(\lambda+\frac{n+2}{n+1}N)v,\\
&\\
h_k\cdot v&=h_k\cdot\sum\limits_{|\ol{m}|=N}E^{\ol{m}}P_{\ol{m}}\\
&=\sum\limits_{|\ol{m}|=N}E^{\ol{m}}(h_k-m_k+\frac{N}{n+1})P_{\ol{m}},\\
&\\
e_{i,k}\cdot v&=e_{i,k}\cdot \sum\limits_{|\ol{m}|=N}E^{\ol{m}}P_{\ol{m}}\\
&=\sum\limits_{|\ol{m}|=N}\prod\limits_{j\neq i}e_{n+2,j}^{m_j}e_{i,k}e_{n+2,i}^{m_i}P_{\ol{m}}\\
&=\sum\limits_{|\ol{m}|=N}\Big(E^{\ol{m}}(e_{i,k}\cdot P_{\ol{m}})-m_i\prod\limits_{j\neq i,k}e_{n+2,j}^{m_j}e_{n+2,k}^{m_k+1}e_{n+2,i}^{m_i-1}P_{\ol{m}}\Big).
\end{align*}
Hence, if $h_{n+2}\cdot v\neq0, h_k\cdot v\neq0, e_{i,k}\cdot v\neq0$, then they are homogeneous of  degree $N$. This completes the proof.
\end{proof}

We also need some results on linear systems. Let $\ol{h_i}=h_i-\delta_{i,n+1}$. For $S\subseteq\{1,2,\cdots,n+1\},\lambda,b\in\bbc$ and $N\in\bb{N}$, we define an $(n+1)\times(n+1)$  matrix $A(\lambda,b,S,N)=(A_{ij}(\lambda,b,S,N))_{1\leq i,j\leq n+1}$ by
\begin{align*}
A_{ij}(\lambda,b,S,N)&=(\delta_{i\in S}+\delta_{i\not\in S}(\ol{h_i}-b))(\delta_{j\in S}(\ol{h_j}-b)+\delta_{j\not\in S}), i\neq j,\\
A_{ii}(\lambda,b,S,N)&=\ol{h_i}-\lambda-N+2.
\end{align*}

\begin{lemma}
\label{detlemma} We have
\[
\det A(\lambda,b,S,N)=(-nb-\lambda-N+1)(b-\lambda-N+2)^n.
\]
\end{lemma}

\begin{proof}
Let $Q_i(x)$ be the matrix obtained from the identity matrix by replacing the $(i,i)$ entry by $x$. For convenience of computations we will use the fraction field $\bbc(\fr{h}_{n})$ of the polynomial ring $\bbc[\fr{h}_{n}]$. Since $\sum\limits_{i=1}^{n+1}h_i=0$, we have
\begin{align*}
\det A(\lambda,b,S,N)&=\prod\limits_{s\not\in S}\det Q_{s}(\frac{1}{\ol{h_{s}}-b})\det A(\lambda,b,S,N)\prod\limits_{s\not\in S}\det Q_{s}(\ol{h_{s}}-b)\\
&=\det(\prod\limits_{s\not\in S}Q_{s}(\frac{1}{\ol{h_{s}}-b})A(\lambda,b,S,N)\prod\limits_{s\not\in S}Q_{s}(\ol{h_{s}}-b))\\
&=\begin{vmatrix}
h_1-\lambda-N+2 & h_2-b & \cdots & h_{n+1}-b-1 \\
h_1-b & h_2-\lambda-N+2 & \cdots & h_{n+1}-b-1 \\
\vdots & \vdots & \ddots & \vdots \\
h_1-b & h_2-b & \cdots & h_{n+1}-\lambda-N+1
\end{vmatrix}\\
\\
&=\begin{vmatrix}
-nb-\lambda-N+1 & h_2-b & \cdots & h_{n+1}-b-1 \\
-nb-\lambda-N+1 & h_2-\lambda-N+2 & \cdots & h_{n+1}-b-1 \\
\vdots & \vdots & \ddots & \vdots \\
-nb-\lambda-N+1 & h_2-b & \cdots & h_{n+1}-\lambda-N+1
\end{vmatrix}\\
\\
&=(-nb-\lambda-N+1)\begin{vmatrix}
1 & h_2-b & \cdots & h_{n+1}-b-1 \\
1 & h_2-\lambda-N+2 & \cdots & h_{n+1}-b-1 \\
\vdots & \vdots & \ddots & \vdots \\
1 & h_2-b & \cdots & h_{n+1}-\lambda-N+1
\end{vmatrix}\\
\\
&=(-nb-\lambda-N+1)\begin{vmatrix}
1 & h_2-b & \cdots & h_{n+1}-b-1 \\
0 & b-\lambda-N+2 & \cdots & 0 \\
\vdots & \vdots & \ddots & \vdots \\
0 & 0 & \cdots & b-\lambda-N+2
\end{vmatrix}\\
&=(-nb-\lambda-N+1)(b-\lambda-N+2)^n.
\end{align*}
\end{proof}

\begin{lemma}
\label{nulllemma1}  If $nb+\lambda+N-1=0$, then for any $\ol{m}\in\bb{Z}_+^{n+1}$,
\[
(a_1^{-1}(m_1+1)\sigma_1^{-1}(\hskip -3pt P_{\ol{m}+\e_1}(S,\a)),\cdots,a_n^{-1}(m_n+1)\sigma_n^{-1}(\hskip -3pt P_{\ol{m}+\e_n}\hskip -3pt(S,\a)),(m_{n+1}+1)\hskip -3ptP_{\ol{m}+\e_{n+1}}\hskip -3pt (S,\a))
\]
is a solution to the linear system
\[
A(\lambda,b,S,N)(X_1,\cdots,X_{n+1})^T=0.
\]
\end{lemma}
\begin{proof}
Following from Lemma~\ref{lemma7}, if $n+1\in S$, then for $1\leq i\leq n$, we have
\begin{align*}
&\sum\limits_{j=1}^{n+1}A_{ij}(\lambda,b,S,N)a_j^{-1}(m_j+1)\sigma_j^{-1}(P_{\ol{m}+\e_j}(S,\a))\\
=&\sum\limits_{\substack{j\in S\\j\neq i}}(\delta_{i\in S}+\delta_{i\not\in S}(h_i-b))(h_j-b-\delta_{j,n+1})P_{\ol{m}}(S,\a)\\
&+\sum\limits_{\substack{j\not\in S\\j\neq i}}(\delta_{i\in S}+\delta_{i\not\in S}(h_i-b))(h_j-b)P_{\ol{m}}(S,\a)\\
&+(h_i-\lambda-N+2)(\delta_{i\in S}+\delta_{i\not\in S}(h_i-b))P_{\ol{m}}(S,\a)\\
=&(\sum\limits_{j\neq i,n+1}(h_j-b)+h_{n+1}-b-1+h_i-\lambda-N+2)(\delta_{i\in S}+\delta_{i\not\in S}(h_i-b))P_{\ol{m}}(S,\a)\\
=&(-nb-\lambda-N+1)(\delta_{i\in S}+\delta_{i\not\in S}(h_i-b))P_{\ol{m}}(S,\a)\\
=&0.
\end{align*}
Also, we have
\begin{align*}
&\sum\limits_{j=1}^{n+1}A_{n+1,j}(\lambda,b,S,N)a_j^{-1}(m_j+1)\sigma_j^{-1}(P_{\ol{m}+\e_j}(S,\a))\\
=&\sum\limits_{\substack{j\in S\\j\neq n+1}}(h_j-b)P_{\ol{m}}(S,\a)+\sum\limits_{j\not\in S}(h_j-b)P_{\ol{m}}(S,\a)+(h_{n+1}-\lambda-N+1)P_{\ol{m}}(S,\a)\\
=&(\sum\limits_{j\neq n+1}(h_j-b)+h_{n+1}-\lambda-N+1)P_{\ol{m}}(S,\a)\\
=&(-nb-\lambda-N+1)P_{\ol{m}}(S,\a)\\
=&0.
\end{align*}

Now suppose that $n+1\not\in S$. Then for $1\leq i\leq n$,
\begin{align*}
&\sum\limits_{j=1}^{n+1}A_{ij}(\lambda,b,S,N)a_j^{-1}(m_j+1)\sigma_j^{-1}(P_{\ol{m}+\e_j}(S,\a))\\
=&\sum\limits_{\substack{j\in S\\j\neq i}}(\delta_{i\in S}+\delta_{i\not\in S}(h_i-b))(h_j-b)\Theta_{\ol{m}}(S,\a)\\
&+\sum\limits_{\substack{j\not\in S\\j\neq i,n+1}}(\delta_{i\in S}+\delta_{i\not\in S}(h_i-b))(h_j-b)\Theta_{\ol{m}}(S,\a)\\
&+(\delta_{i\in S}+\delta_{i\not\in S}(h_i-b))(h_{n+1}-b-1)\Theta_{\ol{m}}(S,\a)\\
&+(h_i-\lambda-N+2)(\delta_{i\in S}+\delta_{i\not\in S}(h_i-b))\Theta_{\ol{m}}(S,\a)\\
=&(\delta_{i\in S}+\delta_{i\not\in S}(h_i-b))(\sum\limits_{j\neq i,n+1}(h_j-b)+h_{n+1}-b-1+h_i-\lambda-N+2)\Theta_{\ol{m}}(S,\a)\\
=&(\delta_{i\in S}+\delta_{i\not\in S}(h_i-b))(-nb-\lambda-N+1)\Theta_{\ol{m}}(S,\a)\\
=&0.
\end{align*}
And
\begin{align*}
&\sum\limits_{j=1}^{n+1}A_{n+1,j}(\lambda,b,S,N)a_j^{-1}(m_j+1)\sigma_j^{-1}(P_{\ol{m}+\e_j}(S,\a))\\
=&\sum\limits_{j\in S}(h_{n+1}-b-1)(h_j-b)\Theta_{\ol{m}}(S,\a)+\sum\limits_{\substack{j\not\in S\\j\neq n+1}}(h_{n+1}-b-1)(h_j-b)\Theta_{\ol{m}}(S,\a)\\
&+(h_{n+1}-\lambda-N+1)(h_{n+1}-b-1)\Theta_{\ol{m}}(S,\a)\\
=&(\sum\limits_{j\neq n+1}(h_j-b)+h_{n+1}-\lambda-N+1)(h_{n+1}-b-1)\Theta_{\ol{m}}(S,\a)\\
=&(-nb-\lambda-N+1)(h_{n+1}-b-1)\Theta_{\ol{m}}(S,\a)\\
=&0.
\end{align*}
This proves the statement in this lemma.
\end{proof}

Similarly, we have

\begin{lemma}
\label{nulllemma2} If $b-\lambda-N+2=0$,  then for any $\ol{m}\in\bb{Z}_+^{n+1}$,
\[
(a_1^{-1}(m_1+1)\sigma_1^{-1}(\hskip -3pt P'_{\ol{m}+\e_1}(S,\a)),\cdots,a_n^{-1}(m_n+1)\sigma_n^{-1}(\hskip -3pt P'_{\ol{m}+\e_n}(S,\a)),(m_{n+1}+1)\hskip -3pt P'_{\ol{m}+\e_{n+1}}(S,\a))
\]
is a solution to the linear system
\[
A(\lambda,b,S,N)(X_1,\cdots,X_{n+1})^T=0.
\]
\end{lemma}

\begin{proof}
\begin{enumerate}[(1)]
\item We first assume that $n+1\in S$.

In this case, for $1\leq i\leq n$, we have
\begin{align*}
&\sum\limits_{j=1}^{n+1}A_{ij}(\lambda,b,S,N)a_j^{-1}(m_j+1)\sigma_j^{-1}(P'_{\ol{m}+\e_j}(S,\a))\\
=&\sum\limits_{\substack{j\in S\\j\neq i,n+1}}(\delta_{i\in S}+\delta_{i\not\in S}(h_i-b))(h_j-b)(h_{n+1}-b-1)\Delta_{\ol{m}}(S,\a)\\
&+(\delta_{i\in S}+\delta_{i\not\in S}(h_i-b))(h_{n+1}-b-1)(h_{n+1}+nb)\Delta_{\ol{m}}(S,\a)\\
&+\sum\limits_{\substack{j\neq S\\j\neq i}}(\delta_{i\in S}+\delta_{i\not\in S}(h_i-b))(h_j-b)(h_{n+1}-b-1)\Delta_{\ol{m}}(S,\a)\\
&+(h_i-\lambda-N+2)(\delta_{i\in S}+\delta_{i\not\in S}(h_i-b))(h_{n+1}-b-1)\Delta_{\ol{m}}(S,\a)\\
=&(\delta_{i\in S}+\delta_{i\not\in S}(h_i-b))(h_{n+1}-b-1)(\sum\limits_{j\neq i,n+1}(h_j-b)+h_i-\lambda-N+2\\
&+h_{n+1}+nb)\Delta_{\ol{m}}(S,\a)\\
=&(\delta_{i\in S}+\delta_{i\not\in S}(h_i-b))(h_{n+1}-b-1)(b-\lambda-N+2)\Delta_{\ol{m}}(S,\a)\\
=&0.
\end{align*}
And
\begin{align*}
&\sum\limits_{j=1}^{n+1}A_{n+1,j}(\lambda,b,S,N)a_j^{-1}(m_j+1)\sigma_j^{-1}(P'_{\ol{m}+\e_j}(S,\a))\\
=&\sum\limits_{\substack{j\in S\\j\neq n+1}}(h_j-b)(h_{n+1}-b-1)\Delta_{\ol{m}}(S,\a)+\sum\limits_{j\not\in S}(h_j-b)(h_{n+1}-b-1)\Delta_{\ol{m}}(S,\a)\\
&+(h_{n+1}-\lambda-N+1)(h_{n+1}+nb)\Delta_{\ol{m}}(S,\a)\\
=&\sum\limits_{j\neq n+1}(h_j-b)(h_{n+1}-b-1)\Delta_{\ol{m}}(S,\a)+(h_{n+1}-\lambda-N+1)(h_{n+1}+nb)\Delta_{\ol{m}}(S,\a)\\
=&-(h_{n+1}+nb)(h_{n+1}-b-1)\Delta_{\ol{m}}(S,\a)+(h_{n+1}-\lambda-N+1)(h_{n+1}+nb)\Delta_{\ol{m}}(S,\a)\\
=&(h_{n+1}+nb)(b-\lambda-N+2)\Delta_{\ol{m}}(S,\a)\\
=&0.
\end{align*}
\item Now we assume that $n+1\not\in S$.

Now for $1\leq i\leq n$, we have
\begin{align*}
&\sum\limits_{j=1}^{n+1}A_{ij}(\lambda,b,S,N)a_j^{-1}(m_j+1)\sigma_j^{-1}(P'_{\ol{m}+\e_j}(S,\a))\\
=&\sum\limits_{\substack{j\in S\\j\neq i}}(\delta_{i\in S}+\delta_{i\not\in S}(h_i-b))(h_j-b)\Upsilon_{\ol{m}}(S,\a)\\
&+\sum\limits_{\substack{j\not\in S\\j\neq i,n+1}}(\delta_{i\in S}+\delta_{i\not\in S}(h_i-b))(h_j-b)\Upsilon_{\ol{m}}(S,\a)\\
&+(\delta_{i\in S}+\delta_{i\not\in S}(h_i-b))(h_{n+1}+nb)\Upsilon_{\ol{m}}(S,\a)\\
&+(h_i-\lambda-N+2)(\delta_{i\in S}+\delta_{i\not\in S}(h_i-b))\Upsilon_{\ol{m}}(S,\a)\\
=&(\delta_{i\in S}+\delta_{i\not\in S}(h_i-b))(\sum\limits_{j\neq i,n+1}(h_j-b)+h_{n+1}+nb+h_i-\lambda-N+2)\Upsilon_{\ol{m}}(s,\a)\\
=&(\delta_{i\in S}+\delta_{i\not\in S}(h_i-b))(b-\lambda-N+2)\Upsilon_{\ol{m}}(S,\a)\\
=&0.
\end{align*}
Also, we have
\begin{align*}
&\sum\limits_{j=1}^{n+1}A_{n+1,j}(\lambda,b,S,N)a_j^{-1}(m_j+1)\sigma_j^{-1}(P'_{\ol{m}+\e_j}(S,\a))\\
=&\sum\limits_{j\in S}(h_{n+1}-b-1)(h_j-b)\Upsilon_{\ol{m}}(S,\a)+\sum\limits_{\substack{j\not\in S\\j\neq n+1}}(h_{n+1}-b-1)(h_j-b)\Upsilon_{\ol{m}}(S,\a)\\
&+(h_{n+1}-\lambda-N+1)(h_{n+1}+nb)\Upsilon_{\ol{m}}(S,\a)\\
=&\sum\limits_{j\neq n+1}(h_{n+1}-b-1)(h_j-b)\Upsilon_{\ol{m}}(S,\a)+(h_{n+1}-\lambda-N+1)(h_{n+1}+nb)\Upsilon_{\ol{m}}(S,\a)\\
=&(h_{n+1}+nb)(b-\lambda-N+2)\Upsilon_{\ol{m}}(S,\a)\\
=&0.
\end{align*}
\end{enumerate}
Hence, the lemma is true.
\end{proof}

Now we can prove our main theorem.

\begin{thm}
\label{mainthm}
Let  $ b\in\bbc,\lambda\in\bbc, \a=(a_1,\cdots,a_n,1)\in(\bbc^*)^{n+1}, S\subseteq \{1,2,...,n+1\}$.
  Then the generalized Verma module $M_{\fr{p}}(V(\a,S,b,\lambda))$ over $\fr{sl}_{n+2}(\bbc)$ is simple if and only if  all the following conditions hold
  \begin{itemize}\item[(i).] $b-\lambda+2\not\in\bb{N}$, \item [(ii).] $nb+\lambda-1\not\in-\bb{N}$, \item[(iii).] $1 \leq |S| \leq n$ or $(n+1)b\notin \bb{Z}_+$.\end{itemize}
\end{thm}

\begin{proof} From Corollary 32 in \cite{Ni1} and from Lemma~\ref{classi} we know that $V(\a,S,b,\lambda)$ is a simple $\fr{a}'$-module if and only if Condition (iii) holds.  Now we assume that this condition holds.

For sufficiency we need to show: if $b-\lambda+2\not\in\bb{N}$ and $nb+\lambda-1\not\in-\bb{N}$, then the induced module $M_{\fr{p}}(V(\a,S,b,\lambda))$ is simple.

 Suppose $W$ is a proper nonzero submodule of $M_{\fr{p}}(V(\a,S,b,\lambda))$ and fix a nonzero element $v \in S$ for which the degree $N$ of $v$ is minimal. From Lemma~\ref{homlemma} we may assume that $v$ is homogeneous. Since we have assumed the simplicity of $V(\a,S,b,\lambda)$ we have $N\geq1$.

Let $v=\sum\limits_{|\ol{m}|=N}E^{\ol{m}}P_{\ol{m}}$. By the minimality of $N$ we must have $e_{i,n+2} \cdot v=0$ for all $1 \leq i \leq n+1$.
By Lemma~\ref{classi} and Lemma~\ref{lemma3}, for any $\ol{m}\in\bb{Z}_+^{n+1}$ with $|\ol{m}|=N-1$ and any $i: 1 \leq i \leq n+1$ we see that
\begin{align*}
0=&\sum_{k\neq i}(m_{k}+1) (e_{i,k} \cdot P_{\ol{m}+\e_{k}})
+(m_{i}+1)(h_{i}-\lambda-N+1)P_{\ol{m}+\e_{i}} \\
=& \sum_{k\neq i}(m_{k}+1)a_ia_k^{-1}(\delta_{i\in S}+\delta_{i\not\in S}(h_i-b-1))(\delta_{k\in S}(h_k-b)+\delta_{k\not\in S})\sigma_{i}\sigma_k^{-1}\cdot (P_{\ol{m}+\e_{k}})\\
&+(m_{i}+1)(h_{i}-\lambda-N+1)P_{\ol{m}+\e_{i}} .
\end{align*}
Hence,
\begin{align*}
 \sum_{k\neq i}&(m_{k}+1)a_k^{-1}\sigma_i^{-1}((\delta_{i\in S}+\delta_{i\not\in S}(h_i-b-1))(\delta_{k\in S}(h_k-b)+\delta_{k\not\in S}))\sigma_k^{-1}\cdot (P_{\ol{m}+\e_{k}})\\
&+a_i^{-1}(m_{i}+1)\sigma_i^{-1}((h_{i}-\lambda-N+1))\sigma_i^{-1}(P_{\ol{m}+\e_{i}}) =0, \forall \; |\ol{m}|=N-1;1 \leq i \leq n+1.
\end{align*}
Then we have
\[
A(\lambda,b,S,N)\begin{pmatrix}
a_1^{-1}(m_1+1)\sigma_1^{-1}(P_{\ol{m}+\e_1})\\
\vdots\\
a_n^{-1}(m_n+1)\sigma_n^{-1}(P_{\ol{m}+\e_n})\\
(m_{n+1}+1)P_{\ol{m}+\e_{n+1}}
\end{pmatrix}=0 \qquad \forall \; |\ol{m}|=N-1.
\]
When $b-\lambda+2\not\in\bb{N}$ and $nb+\lambda-1\not\in-\bb{N}$, from Lemma~\ref{detlemma}, we know that
\[\det A(\lambda,b,S,N)\in\bb{C}^*.\]
Hence,
\[
P_{\ol{m}+\e_i}=0, \qquad \forall \; |\ol{m}|=N-1;1 \leq i \leq n+1.
\]
Therefore, $v=0$. This is a contradiction. Thus, $M_{\fr{p}}(V(\a,S,b,\lambda))$ is simple.

To prove the necessity we need to show that if $nb+\lambda-1\in-\bb{N}$ or $b-\lambda+2\in\bb{N}$, then $M_{\fr{p}}(V(\a,S,b,\lambda))$ is reducible.

First we assume that  $nb+\lambda+N-1=0$ for some $N\in\bb{N}$.  Take
$$v_1=\sum\limits_{|\ol{m}|=N}E^{\ol{m}}P_{\ol{m}}(S,\a).$$

Then for any $1\leq i\leq n+1$, following from Lemma~\ref{lemma3} and Lemma~\ref{nulllemma1}, we have
\begin{align*}
&e_{i,n+2}\cdot v_1\\
=& \sum_{|\ol{m}|=N-1}E^{\ol{m}}  \sum_{k=1}^{n+1}(m_{k}+1)\Big( (1-\delta_{i,k})(e_{i,k} \cdot P_{\ol{m}+\e_{k}}(S,\a))+\delta_{i,k}(h_{i}-\lambda-N+1)\\
&\cdot P_{\ol{m}+\e_{k}}(S,\a)\Big)\\
=&\sum\limits_{|\ol{m}|=N-1}E^{\ol{m}}\Big( \sum_{k\neq i}(m_{k}+1)a_ia_k^{-1}(\delta_{i\in S}+\delta_{i\not\in S}(h_i-b-1))(\delta_{k\in S}(h_k-b)+\delta_{k\not\in S})\cdot\\
&\sigma_{i}\sigma_k^{-1}(P_{\ol{m}+\e_{k}}(S,\a))+(m_{i}+1)(h_{i}-\lambda-N+1)P_{\ol{m}+\e_{i}}(S,\a) \Big)\\
=&a_i\sum\limits_{|\ol{m}|=N-1}E^{\ol{m}}\sigma_i\Big(\sum\limits_{j=1}^{n+1}A_{ij}(\lambda,b,S,N)a_j^{-1}(m_j+1)\sigma_j^{-1}(P_{\ol{m}+\e_j}(S,\a))\Big)\\
=&a_i\sum\limits_{|\ol{m}|=N-1}E^{\ol{m}}\sigma_i(0)\\
=&0.
\end{align*}
Hence, by Lemma~\ref{lemmasub}, $v_1$ generates a proper submodule of $M_{\fr{p}}(V(\a,S,b,\lambda))$. Therefore $M_{\fr{p}}(V(\a,S,b,\lambda))$ is reducible.

Now we assume that  $b-\lambda-N+2=0$ for some $N\in\bb{N}$.  Take
$$
v_2=\sum\limits_{|\ol{m}|=N}E^{\ol{m}}P'_{\ol{m}}(S,\a).
$$

From Lemma~\ref{lemma3} and Lemma~\ref{nulllemma2}, for any $1\leq i\leq n+1$, we have
\begin{align*}
&e_{i,n+2}\cdot v_2\\
=& \sum_{|\ol{m}|=N-1}E^{\ol{m}}  \sum_{k=1}^{n+1}(m_{k}+1)\Big( (1-\delta_{i,k})(e_{i,k} \cdot P'_{\ol{m}+\e_{k}}(S,\a))+\delta_{i,k}(h_{i}-\lambda-N+1)\\
&\cdot P'_{\ol{m}+\e_{k}}(S,\a)\Big)\\
=&\sum\limits_{|\ol{m}|=N-1}E^{\ol{m}}\Big( \sum_{k\neq i}(m_{k}+1)a_ia_k^{-1}(\delta_{i\in S}+\delta_{i\not\in S}(h_i-b-1))(\delta_{k\in S}(h_k-b)+\delta_{k\not\in S})\\
&\cdot\sigma_{i}\sigma_k^{-1}(P'_{\ol{m}+\e_{k}}(S,\a))+(m_{i}+1)(h_{i}-\lambda-N+1)P'_{\ol{m}+\e_{i}}(S,\a) \Big)\\
=&a_i\sum\limits_{|\ol{m}|=N-1}E^{\ol{m}}\sigma_i\Big(\sum\limits_{j=1}^{n+1}A_{ij}(\lambda,b,S,N)a_j^{-1}(m_j+1)\sigma_j^{-1}(P'_{\ol{m}+\e_j}(S,\a))\Big)\\
=&a_i\sum\limits_{|\ol{m}|=N-1}E^{\ol{m}}\sigma_i(0)\\
=&0.
\end{align*}
Hence, Lemma~\ref{lemmasub} implies that $v_2$ generates a proper submodule of $M_{\fr{p}}(V(\a,S,b,\lambda))$. Therefore $M_{\fr{p}}(V(\a,S,b,\lambda))$ is reducible. This completes the proof.
\end{proof}

\begin{rmk} The result in Theorem \ref{mainthm} agrees with that in \cite{McD2} for the case of $\mathfrak{sl}_3$ when $U(\mathfrak{h}_1)$ is a non-degenerate Whittaker module over $\mathfrak{sl}_2$.
\end{rmk}

\begin{rmk} The method to construct simple modules in this paper can be certainly extended to other types of simple Lie algebras. But the computations might be more complicated.
\end{rmk}

\vspace{0.2cm}
 \noindent Y.C.: Wu Wen Tsun Key Laboratory of Mathematics, Chinese Academy of Science, School of Mathematical Sciences, University of Science and Technology of China, Hefei 230026, Anhui, P. R. China.
Email: yatsai@mail.ustc.edu.cn

\vspace{0.2cm} \noindent G.L.: College of Mathematics and Information
Science, Henan University, Kaifeng 475004, China. Email:
liugenqiang@amss.ac.cn

\vspace{0.2cm} \noindent J.N.: Department of Mathematics, Uppsala University, Box 480, SE-751 06, Uppsala, Sweden. Email: jonathan.nilsson@math.uu.se

\vspace{0.2cm} \noindent K.Z.: Department of Mathematics, Wilfrid
Laurier University, Waterloo, ON, Canada N2L 3C5,   and Department of
Mathematics, Xinyang Normal University, Xinyang, Henan,  464000,
P.R. China. Email:
kzhao@wlu.ca

\end{document}